\documentclass[12pt,a4paper]{article}
\usepackage[utf8]{inputenc}
\usepackage[english]{babel}
\usepackage{amsmath}
\usepackage{hhline}
\usepackage{ragged2e}
\usepackage{fancyhdr}
\usepackage{authblk}

\usepackage[margin=2.3cm]{geometry}
\usepackage{parskip}
\usepackage{bbm}
\usepackage{mathrsfs}
\usepackage{amsfonts}
\usepackage{placeins}
\usepackage{graphicx}
\usepackage{subcaption}
\usepackage{hyperref}
\usepackage{tabularx}
\usepackage{caption}
\usepackage{float}
\usepackage{mathtools}
\usepackage{listings}
\usepackage{xcolor}
\usepackage{systeme}
\usepackage{tabularx}
\usepackage{caption}
\usepackage{diagbox}
\usepackage{longtable}
\usepackage{booktabs}
\usepackage{multirow}
\usepackage{siunitx}
\usepackage{enumerate}
\usepackage{amssymb}
\usepackage{amsthm}

\usepackage[square,numbers]{natbib}

\newtheorem{theorem}{Theorem}[section]

\newcommand{\E}{\mathbb{E}}
\newcommand{\R}{\mathbb{R}}
\newcommand{\I}{\mathbb{I}}
\renewcommand{\P}{\mathbb{P}}
\newcommand{\N}{\mathbb{N}}

\providecommand{\keywords}[1]{\small\textbf{\textbf{Keywords:}} #1}

\providecommand{\MSC}[1]{\small \textbf{\textbf{MSC 2020:}} #1}

\title{Kac-Stroock type approximations for the Brownian motion}
\date{}
\author[]{Xavier Bardina\thanks{Corresponding author. \\Both authors are supported by the grant PID2021-123733NB-I00 from SEIDI, Ministerio de Econom\'ia y Competitividad. Salim Boukfal is also supported by the program of predoctoral grants AGAUR-FI (2025 FI-1 00119) Joan Or\'o from the Department of Research and Universities of the Government of Catalonia and the co-funding of the European Social Fund Plus (ESF+).}}
\author[]{Salim Boukfal}

\affil[]{Departament de Matemàtiques, Universitat Autònoma de Barcelona, Cerdanyola del Vallès, Spain}

\affil[]{xavier.bardina, salim.boukfal@uab.cat}

\begin{document}

\maketitle

\begin{abstract}
    In the present paper we show that the processes $X_n = \{X_n(t) \colon t \in [0,1]\}$, $n \in \mathbb{N}$, defined by $X_n(t) = \sqrt{n}C\int_0^t (-1)^{L(nu)} du$, where $L = \{L(t) \colon t \geq 0\}$ is a renewal processes whose inter-arrival times satisfy some integrability conditions and $C > 0$ is some normalizing constant, weakly converge, in the space of continuous functions over $[0,1]$, $\mathcal{C}([0,1])$, to the Brownian motion as $n$ approaches infinity. Thus, generalizing the result of D. W. Stroock in \cite{stroock1982lectures}, where $L$ is taken to be a standard Poisson process. In particular, we see that these results are a mere consequence of Donsker's invariance principle.
\end{abstract}


\MSC{60F05, 60F17, 60G50, 60K05}

\section{Introduction}

It is well know (see \cite{stroock1982lectures}) that, if $N = \{N(t) \colon t \geq 0\}$ is a standard Poisson process, the processes $X_n$ defined by
\begin{equation}\label{kac-stroock}
    X_n(t) = \sqrt{n}\int_0^t (-1)^{N_n(u)}du, \quad t \in [0,1],
\end{equation}
which were introduced by M. Kac in \cite{articlekac} in order to solve the telegrapher's equation, weakly converge, in the space of continuous functions over $[0,1]$ (from now on, $\mathcal{C}$) towards the standard Brownian motion. Since then, several generalizations of this result have been seen. For instance:
\begin{enumerate}
    \item By noting that $-1 = e^{i \theta}$ with $\theta = \pi$ and consider any other angle $\theta \in (0\pi)\cup(\pi,2\pi)$, showing convergence towards the Brownian sheet (see \cite{bardina1}).
    \item Replacing the Poisson process $N$ for any other Lévy process, showing convergence towards the Brownian sheet (see \cite{BardinaRovira2016}).
    \item By considering the multiparameter analogs of the Poisson process and, more generally, of the Lévy Process and showing that the corresponding processes converge towards a Brownian sheet or a complex Brownian sheet (see \cite{BARDINAJuanPabloQuer, BARDINARoviraJolis}).
    \item By considering integrals of functions with respect to such processes and showing convergence towards the stochastic integrals with respect to the Brownian motion or Brownian sheet (see \cite{delgadoJolis, BARDINAJOLISTUDOR, bardinaboukfal1, bardinaboukfal2}).
    \item By modifying the Poisson process or by considering a renewal-reward process (with a very specific reward) in such a way that the obtained convergence is in the strong sense (see \cite{griegoheathmoncayo,BardinaRovirarenewal}).
\end{enumerate}
In order to show the weak convergence in the first four extensions, the authors show that the sequence of stochastic processes is tight and that the finite dimensional distributions or that any weak limit of any subsequence converge as desired.\\
As for the strong convergence in the fifth point, the authors show that, via Skorokhod's embedding theorem, there are some realizations of the processes converging almost surely towards the Brownian motion.

Inspired by the results exhibited in \cite{BardinaRovirarenewal}, in this paper we provide a further extension of the original theorem where the Poisson process $N$ is replaced by a general renewal process whose inter-arrival times verify some integrability condition. Moreover, we show that the obtained convergence (and thus, the convergence when a Poisson process is considered) can be seen as a mere consequence of Donsker's invariance principle and hence, that these processes are closely related to the random walk.

\section{Statement and proof of the main result}

Let $\{U_k \colon k \in \N\}$ be a sequence of i.i.d. non-negative random variables such that $\P\{U_1 = 0\} < 1$ and that $\E[U_1^p] < \infty$ for some $p > 2$. Define $S_0 = 0$ and, for $n \in \N$ and $t \geq 0$,
\begin{equation*}
    S_n = \sum_{k=1}^n U_k, \quad L(t) = \sum_{n=1}^\infty \I_{[0,t]}(S_n),
\end{equation*}
where $\I_A$ is the indicator of the set $A$ (that is, $\I_A(x) = 1$ if $x \in A$ and $\I_A(x)=0$ otherwise). $L = \{L(t) \colon t \geq 0\}$ is usually called a renewal process with arrival times $\{S_n \colon n \in \N\}$ or, equivalently, with inter-arrival times $U_n = S_n - S_{n-1}$, $n \in \N$.

The main result of this paper reads as follows.

\begin{theorem}\label{main theorem}
    The processes $X_n = \{X_n(t) \colon t \in [0,1]\}$, $n \in \N$, defined by
    \begin{equation}\label{processos salim}
        X_n(t) = C\sqrt{n} \int_0^t (-1)^{L(n u)} du,
    \end{equation}
    where $C^2 = \E[U_1]/\text{Var}(U_1)$, weakly converge, in $\mathcal{C}$, towards a standard Brownian motion as $n$ approaches infinity.
\end{theorem}

Before we prove this result, we shall see why the hypotheses required for the proof of Theorem \ref{main theorem} are less restrictive than the ones seen in \cite{BardinaRovirarenewal} and why these are not enough (not a priori, at least) to show the strong convergence of these processes (or a slight modification of them by considering some specific reward) towards the Brownian motion.

We start by noticing that the renewal process $L$ almost surely vanishes in a neighborhood of $0$, meaning that the sample paths of the processes $X_n$ always start increasing. In order to avoid this, in \cite{BardinaRovirarenewal} the authors introduce a sequence of i.i.d. Bernoulli random variables $\{\eta_k \colon k \in \N \cup {0}\}$ of parameter $1/2$ and define the renewal-reward process
\begin{equation*}
    T(t) = \sum_{k=0}^{L(t)}\eta_k,
\end{equation*}
which can take the value $1$ at $t = 0$ implying that the processes
\begin{equation*}\label{processos bardina rovira}
        \Tilde{X}_n(t) = \frac{C}{\sqrt{\beta(n)}} \int_0^t (-1)^{T(u/\beta(n))} du, \quad t \in [0,1],
\end{equation*}
where $\beta \colon \N \to \R$ is some strictly positive sequence (in our case, $\beta(n) = 1/n$) converging to 0 and $C$ is some normalizing constant different from the one in \eqref{processos salim}, need not be increasing at $t=0$.\\
In addition, one can see that the strong convergence displayed in \cite{griegoheathmoncayo, BardinaRovirarenewal} is a consequence of the Borel-Cantelli lemma, for which one needs to impose further conditions on the sequence $\{\beta(n) \colon n \in \N\}$. More precisely, one requires this sequence to be summable, this in particular implies that, contrary to what we will be seeing, the results seen in \cite{BardinaRovirarenewal} need not be valid, for instance, for $\beta(n) = 1/n$ or for some general $\beta(n)$ approaching 0 from above. Furthermore, in order to apply Borel-Cantelli's lemma, the authors also need to make use of some of the bounds provided by the Skorokhod embedding theorem, for which one requires the random variables $U_k$ to have fourth order moments, whilst we only need them to have moments of order $p > 2$.

\begin{proof}[Proof of Theorem \ref{main theorem}]

    Without any loss of generality, we shall assume that $\mu = \E[U_1] > 1$, the general case is a matter of scaling the random variables $U_k$. Indeed, let us assume that we have shown the result for $\mu > 1$, and consider a general sequence $\{U_k \colon k \in \N\}$ verifying the hypotheses for which the theorem holds. For any $\rho \in (0,1)$, the variables $\Tilde{U}_k = U_k/(\rho \mu)$ have mean $1/\rho > 1$, so that the theorem can be applied to the renewal process
    \begin{equation*}
        \Tilde{L}(t) = \sum_{k=1}^\infty \I_{[0,t]}(\Tilde{U}_1 + ... + \Tilde{U}_k) = L(t \rho \mu).
    \end{equation*}
    Hence, the processes
    \begin{equation*}
        \Tilde{X}_n(t) = \sqrt{\frac{n\E[\Tilde{U}_1]}{\text{Var}(\Tilde{U}_1)}} \int_0^t (-1)^{\Tilde{L}(nu)} du = \sqrt{\frac{n}{\rho \text{Var}(U_1)}} \int_0^{t \rho \mu} (-1)^{L(nu)} du, \quad t \in [0,1],
    \end{equation*}
    will converge to a standard Brownian motion, meaning that
    \begin{equation*}
        X_n(t) = \sqrt{\frac{n\E[U_1]}{\text{Var}(U_1)}} \int_0^t (-1)^{L(nu)} du = \sqrt{\rho \mu} \Tilde{X}_n\left( \frac{t}{\rho \mu} \right), \quad t \in [0,1],
    \end{equation*}
    will converge towards a standard Brownian motion as well.
    
    In order to prove the result for $\mu > 1$, we start by noticing that $X_n$ can be written as 
    \begin{equation}\label{xn random sum}
        X_n(t) = \frac{C}{\sqrt{n}} \sum_{j=1}^{L(n t)} (-1)^{j-1}U_j + \frac{C}{\sqrt{n}}(-1)^{L(nt)}\left( nt - S_{L(nt)}\right).
    \end{equation}
    We will show that
    \begin{equation*}
        W_n(t) = \frac{C}{\sqrt{n}} \sum_{j=1}^{L(n t)} (-1)^{j-1}U_j
    \end{equation*}
    and
    \begin{equation*}
        R_n(t) =  \frac{C}{\sqrt{n}}(-1)^{L(nt)}\left( nt - S_{L(nt)}\right)
    \end{equation*}
    converge, respectively, towards a Brownian motion and the null process in the space of càdlàg functions over $[0,1]$ (which we denote by $\mathcal{D}$) with respect to the Skorokhod topology (unless stated otherwise, weak convergence in $\mathcal{D}$ will mean convergence with respect to this topology). If we manage to show this, since $R_n$ will converge to $0$, we will have that $X_n$ will converge towards the Brownian motion in $\mathcal{D}$. Since $X_n$ has continuous paths and the convergence in $\mathcal{D}$ relativized to $\mathcal{C}$ coincides with the weak convergence in $\mathcal{C}$ (with the uniform topology), we will obtain the weak convergence in $\mathcal{C}$ of the processes $X_n$.

    Now, in order to see that $R_n$ converges to $0$ in $\mathcal{D}$, we first note that, for any $t \geq 0$ and $n \in \N$, $S_{L(nt)} \leq nt < S_{L(nt)+1}$, so that
    \begin{equation*}
        |R_n(t)| \leq \Tilde{R}_n(t) \coloneqq \frac{C}{\sqrt{n}} U_{L(nt) + 1}.
    \end{equation*}
    In particular, $\sup_{t} |R_n(t)| \leq \sup_{t}\Tilde{R}_n(t)$. Observe that if we show that $\Tilde{R}_n$ goes to $0$ in $\mathcal{D}$, then so will $R_n$. Indeed, recall that the map $\mathcal{D} \ni x \mapsto \sup_t |x(t)|$ is continuous with respect to the Skorokhod topology, meaning that, by the continuous mapping theorem, $\sup_t \Tilde{R}_n(t)$ will weakly converge (in $\R$) towards 0 as well and thus, since $0$ is a constant, we will have convergence in probability, giving us that, for any $\varepsilon > 0$
    \begin{equation*}
        \P\left\{\sup_{0 \leq t \leq 1}|R_n(t)| \geq \varepsilon \right\} \leq \P\left\{\sup_{0 \leq t \leq 1}\Tilde{R}_n(t) \geq \varepsilon\right\} \xrightarrow{n \to \infty} 0.
    \end{equation*}
    That is, $R_n$ will converge towards $0$ in probability and in the space $\mathcal{D}$ with the uniform topology, which, in turn, will imply convergence in law in the same space with the same topology. However, in $\mathcal{D}$, the topology induced by the uniform metric is finer than the topology induced by the Skorokhod metric (see, for instance, \cite{billingsleyantinc}, pp. 150-151), implying that $R_n$ will converge to $0$ in $\mathcal{D}$ with the Skorokhod topology.\\
    All in all, what we have seen is that it suffices to show that $\Tilde{R}_n$ converges to $0$ as a sequence of processes in $\mathcal{D}$ (with the uniform or the Skorokhod topology).

    To do so, consider the processes$\{\Tilde{\Tilde{R}}_n \colon n \in \N\}$ defined by
    \begin{equation*}
        \Tilde{\Tilde{R}}_n(t) = \frac{C}{\sqrt{n}}U_{[nt]+1}
    \end{equation*}
    and note that
    \begin{equation*}
        \sup_{0 \leq t \leq 1} \Tilde{\Tilde{R}}_n = \frac{C}{\sqrt{n}}\max_{1 \leq i \leq n+1} U_{i}.
    \end{equation*}
    Next, by Jensen's inequality,
    \begin{align*}
       \E&\left[ \max_{1 \leq i \leq n+1} U_{i} \right] = \E\left[ \left(\max_{1 \leq i \leq n+1} U_{i}^p \right)^{\frac{1}{p}} \right] 
       \leq \left( \E\left[ \max_{1 \leq i \leq n+1} U_{i}^p \right]  \right)^{\frac{1}{p}} \\
       &\leq \left( \E\left[ \sum_{i=1}^{n+1} U_{i}^p \right]  \right)^{\frac{1}{p}} 
       = \left( \sum_{i=1}^{n+1}\E\left[ U_{i}^p \right]  \right)^{\frac{1}{p}} 
       = (n+1)^{\frac{1}{p}} \left( \E[U_1^p] \right)^{\frac{1}{p}},
    \end{align*}
    where, in the last step, we have used that the random variables $U_k$ are identically distributed. Since $p > 2$, this means that $\Tilde{\Tilde{R}}_n$ converges in $L^1$ to 0 in $\mathcal{D}$ with the uniform topology. In particular, it will weakly converge in the same space with the same topology.

    Now let us define
    \begin{equation}\label{canvi de temps}
        \Phi_n(t) = \begin{cases}
            \frac{L(nt)}{n}, \quad &\text{if }\frac{L(n)}{n} \leq 1,\\
            \frac{t}{\mu}, \quad &\text{otherwise}.
        \end{cases}
    \end{equation}
    It can be shown, (see the proof of Theorem 17.3 in \cite{billingsleyantinc}, page 149), that $\Phi_n$ weakly converges in $\mathcal{D}_0$ to $\phi$, where $\phi(t) = t/\mu$ and where $\mathcal{D}_0 \subset \mathcal{D}$ is the set of non-decreasing càdlàg functions whose image is contained in $[0,1]$ (in other words, the càdlàg changes of time of the unit interval) with the subspace topology. Under the event $L(n)/n \leq 1$ (the probability of which goes to 1), we have that $\Tilde{R}_n = \Tilde{\Tilde{R}}_n \circ \Phi_n$. Thus, both $\Tilde{R}_n$ and $\Tilde{\Tilde{R}}_n \circ \Phi_n$ have the same limit. Finally, by the continuous mapping theorem (for more details, see \cite[p. 145]{billingsleyantinc}), $\Tilde{\Tilde{R}}_n \circ \Phi_n$ will converge in $\mathcal{D}$ to $0 \circ \phi = 0$, showing that $\Tilde{R}_n$ converges to $0$ in $\mathcal{D}$ and that so does $R_n$.

    The only thing left to do is to show that $W_n$ converges to a standard Brownian motion. Again, this will be a consequence of the convergence in $\mathcal{D}$ of 
    \begin{equation*}
        \Tilde{W}_n(t) = \frac{C}{\sqrt{n}}\sum_{j=1}^{[nt]}(-1)^{j-1}U_j 
    \end{equation*}
    towards the Brownian motion alongside the fact that $W_n$ and $\Tilde{W}_n \circ \Phi_n$ (where $\Phi_n$ the same time change in \eqref{canvi de temps}) have the same limit. 
    
    Even though the random variables $(-1)^{j-1}U_j$ are independent, they are not, in general, identically distributed, so that we cannot apply directly the invariance principle. Nevertheless, observe that
    \begin{equation*}
        \Tilde{W}_n(t) = \frac{C}{\sqrt{n}}\sum_{j=1}^{\left[\frac{[nt]}{2}\right]}\left(U_{2j-1} - U_{2j} \right) + C\frac{1 - (-1)^{[nt]}}{2\sqrt{n}}U_{[nt]}, 
    \end{equation*}
    where the first sum is null if $[nt] < 2$. A similar reasoning to the one used to show that $\Tilde{\Tilde{R}}_n$ converges to $0$ shows that the second term in the latter expression for $\Tilde{W}_n$ converges to $0$ in $\mathcal{D}$ as well. Thus, it remains to show that the processes $\Tilde{\Tilde{W}}_n$ defined by
    \begin{equation*}
        \Tilde{\Tilde{W}}_n(t) =  \frac{1}{\sqrt{n}}\sum_{j=1}^{\left[\frac{[nt]}{2}\right]}\left(U_{2j-1} - U_{2j} \right), \quad t \in [0,1],
    \end{equation*}
    converge towards a Brownian motion. But the latter is an (almost) immediate consequence of Donsker's invariance principle. Indeed, we first note that the random variables $\{U_{2j-1} - U_{2j} \colon j \in \N\}$ are i.i.d. $U_1 - U_2$ and that
    \begin{equation*}
        \E[U_1 - U_2] = 0, \quad \E\left[\left(U_1 - U_2\right)^2\right] = 2\text{Var}(U_1).
    \end{equation*}
    Hence, the random walks
    \begin{equation*}
        B_n(t) = \frac{1}{\sqrt{n}} \sum_{j=1}^{[nt]} (U_{2j-1} - U_{2j})
    \end{equation*}
    will weakly converge in $\mathcal{D}$ towards a Brownian motion with variance $2\text{Var}(U_1)$. However, $\Tilde{\Tilde{W}}_n = B_n \circ \Psi_n$, where $\Psi_n(t) = \frac{[nt]}{2n}$ verifies
    \begin{equation*}
        \sup_{0 \leq t \leq 1} \left|\Psi_n(t) -  \psi(t)\right| \leq \frac{1}{2n} \xrightarrow{n \to \infty} 0,
    \end{equation*}
    where $\psi(t) = {t}/{2}$. Therefore, by the same results in \cite{billingsleyantinc} employed previously, $\Tilde{\Tilde{W}}_n$ weakly converges in $\mathcal{D}$ towards a Brownian motion with variance $\text{Var}(U_1)$ and hence, that $\Tilde{W}_n$ will converge to a Brownian motion with variance $C^2\text{Var}(U_1)$. Since $W_n$ has the same limit as $\Tilde{W}_n \circ \Phi_n$ and the latter converges towards a Brownian motion with variance $C^2\text{Var}(U_1)/\mu = 1$, we obtain the convergence in distribution of $W_n$ in $\mathcal{D}$ towards a standard Brownian motion and, all in all, that $X_n$ weakly converges, in $\mathcal{C}$ with the uniform topology, to the very same process, as was to be shown.
\end{proof}

\bibliographystyle{abbrvnat}
\bibliography{references.bib}

\end{document}